\documentclass[11pt]{amsart}

\usepackage{amssymb}
\usepackage{amsmath}
\usepackage{amsthm}
\usepackage{graphicx,psfrag}
\usepackage{enumerate}

\usepackage{graphicx}
\usepackage[usenames,dvipsnames]{color}

\usepackage[dvips]{hyperref}
\hypersetup{pdfborder={}}

\theoremstyle{plain}
\newtheorem{theorem}{Theorem}[section]
\newtheorem{lemma}[theorem]{Lemma}
\newtheorem{corollary}[theorem]{Corollary}
\newtheorem{proposition}[theorem]{Proposition}

\theoremstyle{definition}
\newtheorem{definition}[theorem]{Definition}

\theoremstyle{remark}

\newtheorem*{remark*}{Remark}

\newtheorem{question}[theorem]{Question}
\newtheorem{conjecture}[theorem]{Conjecture}

\numberwithin{figure}{section}

\newcommand{\Int}{\mathrm{int}}

\begin{document}

\title{Taut foliations in knot complements}
\author{Tao Li}
\address{Department of Mathematics \\
 Boston College \\
 Chestnut Hill, MA 02467}
\email{taoli@bc.edu}
\thanks{Partially supported by an NSF grant}

\author{Rachel Roberts}
\address{Department of Mathematics \\
 Washington University \\
 St. Louis, MO 63130}
\email{roberts@math.wustl.edu}

\date{\today}

\begin{abstract}
We show that for any nontrivial knot in $S^3$, there is an open interval containing zero such that a Dehn surgery on any slope in this interval yields a 3-manifold with taut foliations.  This generalizes a theorem of Gabai on zero frame surgery.
\end{abstract}

\maketitle

\section{Introduction}\label{Sintro}

A transversely orientable codimension one foliation $\mathcal F$ of a 3-manifold $M$ is called {\it taut} \cite{G4} if  every leaf of $\mathcal F$ intersects some closed transverse curve. The existence of a taut foliation in a 3-manifold $M$ provides much interesting topological information about both $M$ and objects embedded in $M$.  If a closed 3-manifold $M$ contains a taut foliation, then either $M$ is finitely covered by $S^2\times S^1$ or $M$ is irreducible \cite{Novikov,Reeb,Rosenberg}.
If a closed 3-manifold $M$ contains a taut foliation, then its fundamental group is infinite \cite{Haefliger, Novikov, GO} and acts nontrivially on interesting 1-dimensional objects (see, for example, \cite{T, CD} and \cite{Palmeira, RSS}), and its universal cover is $\mathbb R^3$ \cite{Palmeira}. Taut foliations can be perturbed to interesting contact structures \cite{ET} and hence can be used to obtain Heegaard-Floer information \cite{OzSz}.
In this paper we seek to add to the understanding of the existence of taut foliations by describing a new  construction of taut foliations.

Let $k$ be a nontrivial knot in $S^3$.  In his proof of the Property R conjecture \cite{G3}, Gabai showed that the knot exterior $M=S^3\setminus int N(k)$ has a taut foliation whose restriction to the torus $\partial M$ is a collection of circles of slope $0$.  Thus a zero frame Dehn surgery on $k$ yields a closed 3-manifold that admits a taut foliation obtained by adding disks along the boundary circles of the taut foliation of $M$.  In this paper, we extend Gabai's theorem from zero frame surgery to any slope in an interval that contains $0$. 
Although we restrict attention to knots in $S^3$, the approach described in this paper applies more generally to manifolds $(M,\partial M)$ with boundary a nonempty union of tori and for which there exists a well-groomed sutured manifold hierarchy which meets each component of $\partial M$ only in essential simple closed curves.

\begin{theorem}\label{Tmain}
Let $k$ be a nontrivial knot in $S^3$. Then there is an interval $(-a,b)$, where $a>0$ and $b>0$, such that for any slope $s\in(-a,b)$, the knot exterior $M=S^3\setminus \Int(N(k))$ has a taut foliation whose restriction to the torus $\partial M$ is a collection of circles of slope $s$.  Moreover, by attaching disks along the boundary circles, the foliation can be extended to a taut foliation in $M(s)$, where $M(s)$ is the manifold obtained by performing Dehn surgery to $k$ with surgery slope $s$. 
\end{theorem}

A group $G$ is called {\it left-orderable} if there is a total order on $G$ which is invariant under left multiplication. We thank Liam Watson for calling our attention to the following results.

\begin{corollary}\label{liam}
Let $k$ be a hyperbolic knot in $S^3$ and let $M(1/n)$ denote the manifold obtained by $1/n$ Dehn filling along $k$. Then there is some number $N=N(k)$ such that  $\pi_1(M(1/n))$ is left-orderable whenever $|n|>N$.
\end{corollary}

\begin{proof} 
The surgered manifold $M(1/n)$ is a homology $S^3$ and, by 
Thurston's hyperbolic Dehn surgery theorem \cite{Th0},  atoroidal when $|n|$ is sufficiently large
 (or, equivalently, when $1/n$ is sufficiently small).
Moreover,  by Theorem~\ref{Tmain},
$M(1/n)$ 
  contains a transversely oriented taut foliation whenever $1/n$ is sufficiently close to $0$. It therefore follows from
Corollary 7.6 of \cite{CD}  that $\pi_1(M(1/n))$ is left-orderable. 
\end{proof}

Ozsv\'ath and  Szab\'o  defined the Heegaard Floer homology group $\widehat{HF}(Y)$ of a 3-manifold $Y$ in \cite{OzSz1,OzSz2}. In \cite{OzSz3}, they define  L-space as follows.

\begin{definition} (Definition 1.1, \cite{OzSz3})
A closed three-manifold is called an {\it L-space} if $H_1(Y;\mathbb Q)=0$ and $\widehat{HF}(Y)$ is  free abelian group of rank $|H_1(Y;\mathbb Z)|$.
\end{definition}

\noindent
L-spaces are therefore the closed 3-manifolds with the simplest possible Heegaard Floer homology groups and the following is an important open question:
\begin{question} (Question 11, \cite{OzSz4})
Is there a topological characterization of L-spaces (i.e. which makes no reference to Floer homology)?
\end{question}
\noindent
Recently, Hedden and Levine proposed the following partial answer to this question.
\begin{conjecture} (Conjecture 1, \cite{HL})\label{HL}
If $Y$ is an irreducible homology sphere that is an L-space, then $Y$ is  homeomorphic to either $S^3$ or the Poincar\'{e} homology sphere.
\end{conjecture}
\noindent
Approaches to understanding L-spaces have included investigations into the following two questions. 
Are  L-spaces exactly those irreducible rational homology 3-spheres which contain no transversely oriented taut foliation? Are L-spaces exactly those irreducible rational homology 3-spheres which have non-left-orderable fundamental groups?  (See   \cite{BGW} for a nice survey.)

Recently, Boyer, Gordon and Watson made the following conjecture.

\begin{conjecture} (Conjecture 1, \cite{BGW})\label{gbw}
An irreducible rational homology 3-sphere is an L-space if and only if its fundamental group is not left-orderable.
\end{conjecture}
\noindent
With Conjecture~\ref{gbw} in mind, we compare  Corollary~\ref{liam} with the following result, which appears in various contexts (Corollary 1.3 \cite{OzSz}, Corollary 1.5 \cite{Gh}), but is stated most conveniently  as Proposition 5 in \cite{HW}.

\begin{proposition} (\cite{OzSz}, \cite{HW})\label{hw}
Suppose $k$ is a nontrivial knot in $S^3$ and let $M(1/n)$ denote the manifold obtained by $1/n$ Dehn filling along $k$. If $M(1/n)$ is an L-space, then either $n=1$ and $k$ is the right-handed trefoil or $n=-1$ and $k$ is the left-handed trefoil.
\end{proposition}
\noindent
It follows that Conjecture~\ref{HL} holds for 3-manifolds obtained by surgery on knots in $S^3$. And
it follows from Corollary~\ref{liam} and Proposition~\ref{hw} that Conjecture~\ref{gbw} holds for 3-manifolds obtained by $1/n$ surgery on the complement of hyperbolic knots when $|n|$ is sufficiently large.

In Theorem~\ref{Tmain}, the interval $(-a,b)$ depends both on the knot $k$ and on the sutured manifold decomposition in \cite{G3}.  In \cite{R1, R2}, it is shown that if $k$ is a fibered hyperbolic knot (not necessarily in $S^3$), then this interval can always be chosen to contain $(-1,\infty), (-\infty,1)$, or $(-\infty,\infty)$. 
Moreover, the values of $a$ and $b$ in a maximal such interval $(-a,b)$ reveal information about the pseudo-Anosov monodromy and hence the
geometry of $M$.


\begin{question} 
Let $k$ be a nontrivial knot in $S^3$. What is the maximal interval $(-a,b)$, where $a>0$ and $b>0$, such that for any slope $s\in(-a,b)$, the knot exterior $M=S^3\setminus \Int(N(k))$ has a taut foliation whose restriction to the torus $\partial M$ is a collection of circles of slope $s$ and such that, by attaching disks along the boundary circles, the foliation can be extended to a taut foliation in $M(s)$, where $M(s)$ is the manifold obtained by performing Dehn surgery to $k$ with surgery slope $s$?
\end{question}

\begin{conjecture}
Such a maximal interval will always contain $(-1,1)$.
\end{conjecture}

The proof of the main theorem uses theorems in \cite{L, L1} on branched surfaces to generalize the approach
of \cite{R1} to nonfibred knots.  We first use Gabai's sutured manifold decomposition  \cite{G1,G2,G3} to construct a branched surface $B$.  Then, after first splitting $B$ as necessary,  we add in some product disks to get a new branched surface that carries more laminations which extend to taut foliations.  The key point in the construction is to add branch sectors so that the new branched surface does not contain any sink disk.  By \cite{L, L1}, this means that the branched surface carries a lamination.

\section{Laminar branched surfaces}



\begin{definition}
A \emph{branched surface} $B$ in $M$ is a union of finitely many compact smooth surfaces gluing together to form a compact subspace (of $M$)  locally modeled on Figure~\ref{Fbranch}(a) (ignore the arrows in the picture for now). 
\end{definition}

\begin{figure}
\begin{center}
\psfrag{(a)}{(a)}
\psfrag{(b)}{(b)}
\psfrag{horizontal}{$\partial_hN(B)$}
\psfrag{v}{$\partial_vN(B)$}
\includegraphics[width=4.0in]{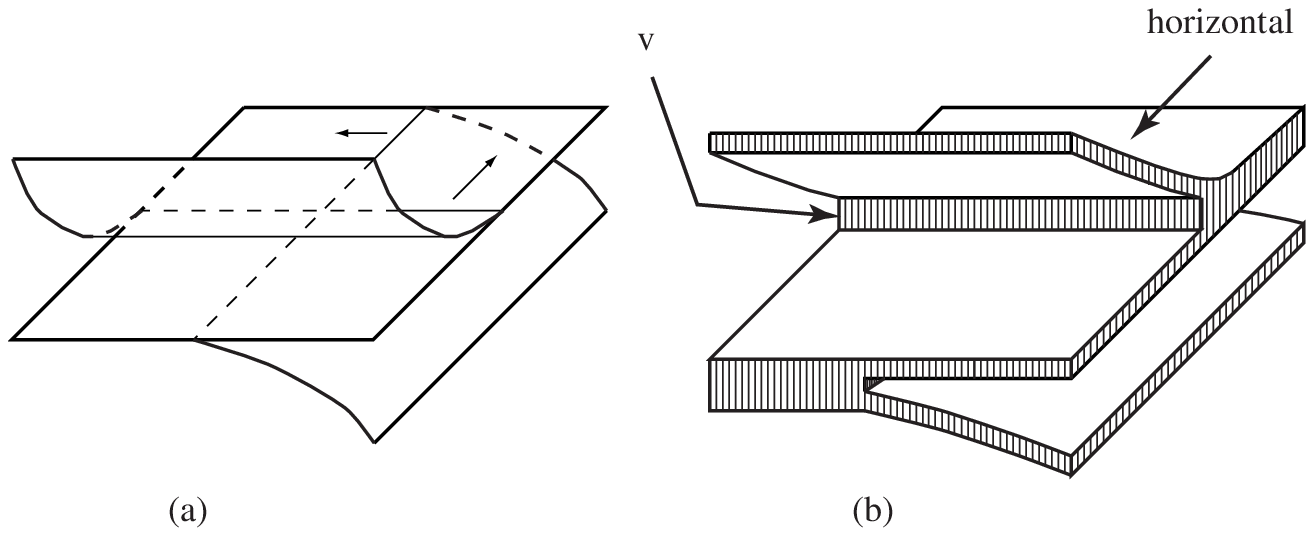}
\caption{}\label{Fbranch}
\end{center}
\end{figure}

Given a branched surface $B$ embedded in a 3-manifold $M$, we denote by $N(B)$ a regular neighborhood of $B$, as shown in Figure~\ref{Fbranch}(b).  One can regard $N(B)$ as an interval bundle over $B$.  We denote by $\pi : N(B)\to B$ the projection that collapses every interval fiber to a point.  As shown in Figure~\ref{Fbranch}(b), the boundary of $N(B)$ consists of two parts: the horizontal boundary $\partial_hN(B)$ which is transverse to the $I$-fibers of $N(B)$, and the vertical boundary $\partial_vN(B)$ which is the union of subarcs of the $I$-fibers.  The \emph{branch locus} of $B$ is $L=\{b\in B:$ $b$ does not have a neighborhood in $B$ homeomorphic to $\mathbb{R}^2  \}$.  We call the closure (under the path metric) of each component of $B\setminus L$ a \emph{branch sector} of $B$.  $L$ is a collection of smooth immersed curves in $B$.  Let $Z$ be the union of double points of $L$.  We associate with every component of $L\setminus Z$ a normal vector (in $B$) pointing in the direction of the cusp, as shown in Figure~\ref{Fbranch}(a).     We call it the \emph{branch direction} of this arc.  Let $D$ be a disk branch sector of $B$.  We call $D$ a \emph{sink disk} if the branch direction of every smooth arc in its boundary points into the disk and $D\cap\partial M=\emptyset$.  We call $D$ a \emph{half sink disk} if $\partial D\cap\partial M\ne\emptyset$ and the branch direction of each arc in $\partial D\setminus \partial M$ points into $D$.  Note that $\partial D\cap\partial M$ may not be connected.

Laminar branched surfaces were introduced in \cite{L} as a branched surface with the usual properties in \cite{GO} plus a condition that there is no sink disk.  The notion of laminar branched surface was slightly extended to branched surfaces with boundary, by adding a requirement that there is no half sink disk \cite{L1}.  Note that if a branched surface has no half sink disk, then one can arbitrarily split the branched surface near its boundary train track without creating any sink disk.  This plus the main theorem of \cite{L} implies the following theorem in \cite{L1}.  Note that the condition that there is no sink disk basically guarantees that the branched surface carries a lamination and the other conditions in \cite{GO} imply that the lamination is an essential lamination.

\begin{theorem}[Theorem 2.2 in \cite{L1}]\label{T1}
Let $M$ be an irreducible and orientable 3-manifold whose boundary is an incompressible torus.  Suppose $B$ is a laminar branched surface and $\partial M\setminus\partial B$ is a union of bigons.  Then, for any rational slope $s\in\mathbb{Q}\cup\infty$ that can be realized by the train track $\partial B$, if $B$ does not carry a torus that bounds a solid torus in $M(s)$, then $B$ fully carries a lamination $\mathcal{L}$ whose boundary consists of loops of slope $s$ and $\mathcal{L}$ can be extended to an essential lamination in $M(s)$.
\end{theorem}

\section{Sutured manifold decompositions}
In \cite{G1}, Gabai introduced the notions of sutured manifold and sutured manifold decomposition.  We will state  basic definitions and theorems as needed for this paper
but we refer the reader to \cite{G1,G2,G3} for a more detailed description.  The papers \cite{A,C2C,J}  and book \cite{C1C} also provide nice descriptions of some of Gabai's sutured manifold theory.  In this paper, we will use branched surfaces to describe sutured manifolds and sutured manifold decompositions.  

\begin{definition}[Definition 2.6, Gabai \cite{G1}]\label{DG26}
A {\it sutured manifold} $(M,\gamma)$ is a compact oriented 3-manifold $M$ together with a set $\gamma\subset \partial M$ of pairwise disjoint annuli $A(\gamma)$ and tori $T(\gamma)$. Furthermore, the interior of each component of $A(\gamma)$ contains a {\it suture}; i.e., a homologically nontrivial oriented simple closed curve. We denote the set of sutures by $s(\gamma)$.

Finally, every component of $R(\gamma)=\partial M\setminus \Int(\gamma)$ is oriented. Define $R_+(\gamma)$ (or $R_-(\gamma)$) to be those components of 
$\partial M\setminus \Int(\gamma)$ whose normal vectors point out of (into) $M$. The orientations on $R(\gamma)$ must be coherent with respect to $s(\gamma)$; i.e., if $\delta$ is a component of $\partial R(\gamma)$ and is given the boundary orientation, then $\delta$ 
must represent the same homology class in $H_1(\gamma)$ as some suture.
\end{definition}

Roughly speaking, a sutured manifold is a 3-manifold together with extra information about $\partial M$. 
Given a sufficiently nice surface embedded in a sutured manifold $(M,\gamma)$, it is important to be able to cut $M$ open along $S$ while keeping track of
corresponding  boundary information. This is captured in the following definition.

\begin{definition}[Definition 3.1, Gabai \cite{G1}] 
Let $(M,\gamma)$ be a sutured manifold and $S$ a properly embedded surface in $M$ such that for every component $\lambda$ of $S\cap \gamma$ one of (1)-(3) holds:
\begin{enumerate}
\item $\lambda$ is a properly embedded nonseparating arc in $\gamma$.
\item $\lambda$ is a simple closed curve in an  annular component $A$ of $\gamma$ in the same homology class as $A\cap s(\gamma)$.
\item $\lambda$ is a homotopically nontrivial curve in a toral component $T$ of $\gamma$, and if $\delta$ is another component of $T\cap S$, then $\lambda$ and $\delta$ represent the same homology class in $H_1(T)$.
\end{enumerate}

The surface $S$ defines a {\it sutured manifold decomposition} $$(M,\gamma)\overset{S}{\rightsquigarrow} (M',\gamma')$$ where
$$M'=M\setminus \Int(N(S))$$ and 
\begin{eqnarray*}
\gamma' & = & (\gamma\cap M')\cup N(S'_+\cap R_-(\gamma))\cup N(S'_-\cap R_+(\gamma)),\\
R'_+(\gamma') & = & ((R_+(\gamma)\cap M')\cup S'_+)\setminus \Int(\gamma'),  \mbox{and}\\
R'_-(\gamma') & = & ((R_-(\gamma)\cap M')\cup S'_-)\setminus \Int(\gamma')
\end{eqnarray*}
where $S'_+$ ($S'_-$) is that component of $\partial N(S)\cap M'$ whose normal vector points out of (into) $M'$.
\end{definition}

\begin{definition}[Definition 0.2, Gabai \cite{G2}] 
A sutured manifold decomposition
$$(M,\gamma)\overset{S}{\rightsquigarrow} (M',\gamma')$$
is called {\it well-groomed} if for each component $V$ of $R(\gamma)$, $S\cap V$ is a union of parallel, coherently oriented, nonseparating 
closed curves and arcs.
\end{definition}

\begin{definition}[Definition 3.2, Gabai \cite{G3}] \label{DG32}
Let $$(M,\partial M)\overset{S_1}{\rightsquigarrow} (M_1,\gamma_1)\overset{S_2}{\rightsquigarrow}\cdots\overset{S_n}{\rightsquigarrow}(M_n, \gamma_n)$$ 
be a sequence of sutured manifold decompositions where $\partial M$ is a nonempty union of tori. Define $E_0=\partial M$. Define $E_i$ to be the union of those components of $E_{i-1}\setminus \Int(N(S_i))$ which are annuli and tori (i.e., if $M_i$ is viewed as a submanifold of $M$, then $E_i$ consists of those components of $\gamma_i$ which are contained in $\partial M$). The components of $E_i$ are called the {\it boundary sutures} of $\gamma_i$.
\end{definition}

We also introduce the following definition.

\begin{definition}
Let $(M,\gamma)$ and $(N,\tau)$ be sutured manifolds. 
We will call $(M,\gamma)$ a {\it sutured submanifold} of $(N,\tau)$, and write $(M,\gamma)\subset (N,\tau)$,  if
$M$ is  a union  of components of $N$ and $\gamma=\tau\cap M$.

If $(M,\gamma)\subset (N,\tau)$,  then $(N,\tau)\setminus (M,\gamma)$ denotes the sutured manifold $(N\setminus M, \tau\setminus\gamma)$.
\end{definition}

\begin{theorem}[Gabai, Lemmas 3.6 and 5.1 in \cite{G3}] \label{Gabai}
Let $k$ be a knot in $S^3$. There is a well-groomed sutured manifold sequence
$$(M,\gamma)\overset{S_1}{\rightsquigarrow} (M_1,\gamma_1)\overset{S_2}{\rightsquigarrow}\cdots\overset{S_n}{\rightsquigarrow}(M_n, \gamma_n)=(S\times I,\partial S\times I)$$ of $$(M,\gamma) = (S^3\setminus \Int(N(k)),\partial N(k))$$ such that $\partial S_i\cap\partial N(k)$ is a (possibly empty)  union of circles for each $i,1\le i\le n$, $S_1$ is a minimal genus Seifert surface, and $S$ is a compact (not necessarily connected) oriented surface.
\end{theorem}

\subsection{Sutured manifold decompositions determine branched surfaces}

As described
by Gabai as Construction 4.6  in \cite{G3} (and detailed further in \cite{C2C}), a sutured manifold decomposition sequence corresponds to building a (finite depth) branched surface, starting with $S_1$ and successively adding the $S_i$'s.  To see this,  inductively construct a sequence of transversely oriented branched surfaces.  Let $B_1=S_1$. So we may view $M_1=M\setminus \Int(N(B_1))$, where $N(B_1)$ is a fibered neighborhood of $B_1$.  As a sutured manifold $(M_1,\gamma_1)$, its suture  $\gamma_1$ is the annulus $\overline{\partial M\setminus N(B_1)}$ and the two components of $\partial_hN(B_1)$ are the plus and minus boundaries $R_+(\gamma_1)$ and $R_-(\gamma_1)$ of the sutured manifold.  We may view $R_+(\gamma_1)$ and $R_-(\gamma_1)$ as lying on the plus and minus sides of $S_1$ respectively and we assign a normal direction for $B_1=S_1$ pointing from the plus side to the minus side.  

Suppose we have constructed a branched surface $B_k$ using the surfaces $S_1,\dots S_k$ in the sutured manifold decomposition, such that $M\setminus \Int(N(B_k))=M_k$ and the suture $\gamma_k$ of $(M_k,\gamma_k)$ consists of $\partial_vN(B_k)$ and a collection of annuli in the boundary torus $\partial M$.  Now we consider the sutured manifold decomposition $(M_k,\gamma_k)\overset{S_{k+1}}{\rightsquigarrow} (M_{k+1},\gamma_{k+1})$.  The surface $S_{k+1}$ has a normal vector.  Then we can deform $B_k\cup S_{k+1}$ into a branched surface $B_{k+1}$ as follows:
\begin{enumerate}
  \item  for each component of $\partial S_{k+1}$ that is not totally inside $\partial_vN(B_k)$, we can deform $B_k\cup S_{k+1}$ near $\partial S_{k+1}$ as in Figure~\ref{Fsuture}(a), so that the normal directions of $B_k$ and $S_{k+1}$ are compatible in the newly constructed branched surface. 
  \item for each component $c$ of $\partial S_{k+1}$ lying inside a suture $\partial_vN(B_k)$, we first slightly isotope $S_{k+1}$ by pushing $c$ into $R_\pm(\gamma_k)\subset\partial_hN(B_k)$, then as shown in Figure~\ref{Fsuture}(b), we can deform $B_k\cup S_{k+1}$ near $c$ into a branched surface.  By the requirement of the normal directions in the sutured manifold decomposition, the normal directions of $B_k$ and $S_{k+1}$ are compatible in the newly constructed branched surface. 
\end{enumerate}  
 It follows from the definition of sutured manifold decomposition \cite{G1} that $M\setminus \Int(N(B_{k+1}))=M_{k+1}$ and the suture $\gamma_{k+1}$ of $(M_{k+1},\gamma_{k+1})$ consists of $\partial_vN(B_{k+1})$ and a collection of annuli in the boundary torus $\partial M$. 
We will sometimes use the notation $$B_{k+1} = B_{(M_{k+1},\gamma_{k+1})} = B_{\langle S_1; S_2; ... ; S_{k+1}\rangle}$$ 

In summary,  there is a map \\

\begin{eqnarray*}
\{\mbox{sutured manifold decomposition sequences}\} & \to & \{\mbox{properly embedded branched surfaces}\}\\
(S_1,S_2,...,S_l) & \mapsto &  B_{\langle S_1; S_2; ... ; S_{l}\rangle}
\end{eqnarray*}
and a (forgetful) map
\begin{align*}\{\mbox{properly embedded branched surfaces}\}\to\{\mbox{sutured 3-manifolds}\}\\
B\mapsto (M_B,\gamma_B)=(M\setminus \Int(N(B)),\partial_vN(B)\cup E')
\end{align*}
where $E'\subset\partial M$ satisfies $E'=E$, the set of boundary sutures,  if $B$ intersects $\partial M$ only in longitudes. For future reference, it is useful to highlight that under this
correspondence, $\partial_hN(B)$ corresponds naturally to $R_+(\gamma_B)\cup R_-(\gamma_B)$.

\begin{figure}
\begin{center}
\psfrag{a}{(a)}
\psfrag{b}{(b)}
\psfrag{B}{$B_k$}
\psfrag{S}{$S_{k+1}$}
\includegraphics[width=4.5in]{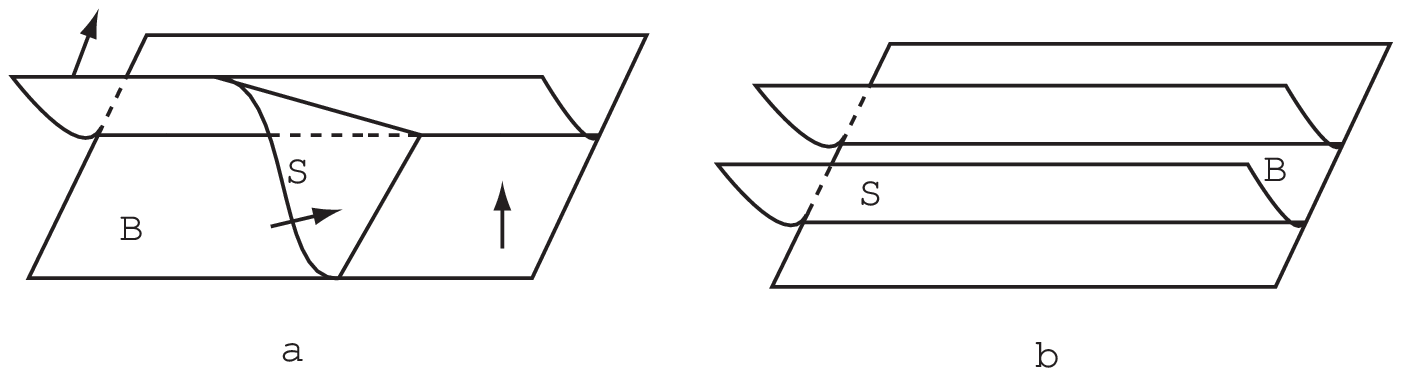}
\caption{}\label{Fsuture}
\end{center}
\end{figure}

\section{The construction}

\subsection{Modifying the sutured manifold hierarchy} Given a well-groomed sutured manifold hierarchy satisfying the conclusions of Theorem~\ref{Gabai},
we can inductively construct the sequence of branched surfaces $B_1,\dots, B_n$ corresponding to the sutured manifold decomposition.  The branched surface $B_n$ in the end has the properties that (1) $M\setminus \Int(N(B_n))$ is a product and (2) $\partial B_n$ is a collection of circles in $\partial M$ of slope $0$.  In particular, any taut foliation carried by $B_n$ will also necessarily meet $\partial M$ only in
simple closed curves of slope $0$.

To obtain a branched surface carrying taut  foliations realizing an open interval of boundary slopes about $0$, it is necessary to modify the sutured manifold hierarchy, or, equivalently, the sequence of branched surfaces $B_k$. In this section, we describe one way of doing this. We break the process into two steps.

As a first step, we slightly modify the sutured manifold hierarchy by adding some parallel copies of the surfaces $S_k$. Equivalently, we modify the sequence of branched surfaces $B_k$ by adding some parallel copies of the surfaces $S_k$.  This operation is equivalent to a splitting of the branched surface. 
As a second (and final) step, we further modify the sutured manifold hierarchy by adding carefully chosen product disks.

Before giving a precise description of these steps, we introduce some terminology.
Let $B$ be a transversely oriented branched surface and let $F$ be a component of $\partial_hN(B)$.  The boundary of $F$ has two parts: $\partial F\cap\partial M$ and $\partial F\cap\partial_vN(B)$.  We call $\partial F\cap\partial_vN(B)$ the \emph{internal boundary} of $F$.  Let $L$ be the branch locus of $B$.  Let $L_F$ be the closure of $\pi^{-1}(L)\cap\Int(F)$, where $\pi\colon N(B)\to B$ is the map collapsing each interval fiber to a point.  So $L_F$ is a trivalent graph properly embedded in $F$. We call $L_F$ the \emph{projection} of the branch locus to $F$.  Each arc in $L_F$ has a normal direction induced from the branch direction of $L$.  

\begin{definition}\label{Dgood}
Let $F$ be a component of $\partial_hN(B)$ with $\partial F\cap\partial M\ne\emptyset$ and let $\eta$ be an arc properly embedded in $F$.  If $F$ has nonempty internal boundary, we require that $\eta$ connects $\partial F\cap\partial M$ to the internal boundary of $F$.  Choose $\eta$ so that it intersects $L_F$ transversely  and only at points in the interior of edges of $L_F$ (namely, it misses all triple points). Since $\eta$ is transverse to $L_F$, the induced branch direction of $L_F$ gives a direction along $\eta$ for each point in $\eta\cap L_F$.  
We say $\eta$ is \emph{good} if these induced directions are coherent along $\eta$ and
 all point away from an endpoint of $\eta$ that lies in $\partial M$.  

We say $F$ is \emph{good} if $F$  satisfies the following properties: 
\begin{enumerate}
  \item  the closure of each component $D$ of $F\setminus L_F$ has a boundary arc with induced branch direction (from $L_F)$ pointing out of $D$,
  \item  if $F$ has internal boundary, then there is a set of disjoint good arcs, denoted by $\Gamma_F$, connecting each component of $\partial F\cap\partial M$ to the internal boundary of $F$, 
  \item if $F$ has no internal boundary (in which case, $F$ must be a Seifert surface of the knot exterior), then there is a properly embedded nonseparating good arc in $F$, which we also denote by $\Gamma_F$.  
\end{enumerate}  
\end{definition}

\begin{lemma}\label{L1}
Let $B$ be a branched surface.  If each component of $\partial_hN(B)$ is good, then $B$ does not contain any sink disk or half sink disk.
\end{lemma}
\begin{proof}
Let $F$ be a component of $\partial_hN(B)$ and let $L_F$ be as above.  Let $P$ be the closure (under path metric) of a component of $F\setminus L_F$.  So $P$ can be viewed as a copy of a branch sector of $B$.  It follows from part (1) of Definition~\ref{Dgood} that  $B$ has no sink disk or half sink disk.
\end{proof}

\begin{definition}
We say the branched surface $B$ is \emph{good}  if $B$ satisfies the following properties:
\begin{enumerate}
\item every component of $\partial_hN(B)$ is good, and 
\item the arc systems $\Gamma_F$ as described in (2) and (3) above can be chosen  so that the projections $\pi(\Gamma_F)$, as $F$ ranges over all components of  $\partial_hN(B)$, are disjoint in $B$.
\end{enumerate}
\end{definition}
\noindent Note that these good arcs $\Gamma_F$ will be the arcs along which we will attach product disks.

\subsection{Step 1: Splitting $B_n$}

Next we will  describe the first modification of  a sutured manifold decomposition sequence  satisfying the conclusions of Theorem~\ref{Gabai}.

\begin{lemma}\label{L2}
Let $k$ be a nontrivial knot in $S^3$ and $M=S^3\setminus  \Int(N(k))$ the knot exterior.  
Let 
$$(M,\partial M)\overset{S_1}{\rightsquigarrow} (M_1,\gamma_1)\overset{S_2}{\rightsquigarrow}\cdots\overset{S_n}{\rightsquigarrow}(M_n, \gamma_n)=(S\times I,\partial S\times I)$$ be a well-groomed sutured manifold hierarchy satisfying the conclusions of Theorem~\ref{Gabai}.
Then there exists a well-groomed sutured manifold hierarchy 
$$
(M,\gamma)\overset{S_1}{\rightsquigarrow} (M^{\prime}_1,\gamma^{\prime}_1)\overset{R_1^{\prime}}{\rightsquigarrow} (M^{\prime\prime}_1,\gamma^{\prime\prime}_1)\overset{S_2}{\rightsquigarrow}(M^{\prime}_2,\gamma^{\prime}_2)\overset{R_2^{\prime}}{\rightsquigarrow} (M^{\prime\prime}_2,\gamma^{\prime\prime}_2)\overset{S_3}{\rightsquigarrow}\cdots\overset{S_n}{\rightsquigarrow}(M^{\prime}_n, \gamma^{\prime}_n)
$$
 which also satisfies the conclusions of Theorem~\ref{Gabai}.
Moreover, the branched surfaces $B_l^{\prime}=B_{(M^{\prime}_l,\gamma_l^{\prime})}, 1\le l\le n$, satisfy
the conditions: 
\begin{enumerate}
  \item $\partial B_l'\cap\partial M$ is a collection of simple closed curves of slope 0 in $\partial M$ for each $l$.
  \item $(M_l,\gamma_l)$ is a sutured submanifold of $(M^{\prime}_l, \gamma^{\prime}_l)$ and $(M^{\prime}_l, \gamma^{\prime}_l)\setminus (M_l,\gamma_l)$ is  a product sutured manifold for each $l$.
  \item every branched surface $B_l'$ is good.
  \item no $B_l'$ carries a torus.
\item $(M^{\prime}_n, \gamma^{\prime}_n)$ is a product sutured manifold $(S^{\prime}\times I, \partial S^{\prime}\times I)$.
\end{enumerate}  
\end{lemma}
\begin{proof}
First note that, in the sutured manifold hierarchy above, each $R_i'$ is a parallel copy of some components of $R_+(\gamma_i')\cup R_-(\gamma_i')$.

We proceed by induction on $l$.
Since $k$ is nontrivial and hence $S_1$ has genus at least one, the branched surface $B_1'=S_1$ is easily seen to satisfy conditions (1)-(4). So suppose we have constructed 
$$(M,\gamma)\overset{S_1}{\rightsquigarrow} (M^{\prime}_1,\gamma^{\prime}_1)\overset{R_1^{\prime}}{\rightsquigarrow} (M^{\prime\prime}_1,\gamma^{\prime\prime}_1)\overset{S_2}{\rightsquigarrow}(M^{\prime}_2,\gamma^{\prime}_2)\overset{R_2^{\prime}}{\rightsquigarrow} (M^{\prime\prime}_2,\gamma^{\prime\prime}_2)\overset{S_3}{\rightsquigarrow}\cdots\overset{S_l}{\rightsquigarrow}(M^{\prime}_l, \gamma^{\prime}_l)$$ satisfying the conclusions of Theorem~\ref{Gabai}
and  such that the corresponding branched surfaces  $B_i'=B_{(M_i',\gamma_i')}$
satisfy the  conditions (1)-(4) for all $i$, 
$1\le i\le l$.

By condition (2), we may consider $(M_l,\gamma_l)$ to be a sutured submanifold of $(M'_l,\gamma'_l)$.
Let $R'_+(\gamma_l)$ and $R'_-(\gamma_l)$ be parallel copies of $R_+(\gamma_l)$ and $R_-(\gamma_l)$, chosen to be properly embedded in $(M_l,\gamma_l)\subset (M'_l,\gamma'_l)$ and
with boundary lying in $E_l\cup A(\gamma_l)$ (see Definition~\ref{DG26} and Definition~\ref{DG32}). Set $R'_l=R'_+(\gamma_l)\cup R'_-(\gamma_l)$.  We first consider the sutured manifold decomposition $(M^{\prime}_l,\gamma^{\prime}_l)\overset{R_l^{\prime}}{\rightsquigarrow} (M^{\prime\prime}_l,\gamma^{\prime\prime}_l)$.  By the definition of $R_l'$, this decomposition only adds some product complementary regions.  Set $B_l^{\prime\prime}=B_{(M_l^{\prime\prime},\gamma_l^{\prime\prime})}$.  The change from $B_l'$ to $B_l''$ is basically adding branch sectors corresponding to $R_l'$ and this operation creates some product complementary regions, see Figure~\ref{Ftype}(a) for a schematic picture.  We may view $(M_l,\gamma_l)\subset (M^{\prime\prime}_l,\gamma^{\prime\prime}_l)$, and consider the sutured manifold decompositions
$$ (M^{\prime}_l,\gamma^{\prime}_l)\overset{R_l^{\prime}}{\rightsquigarrow} (M^{\prime\prime}_l,\gamma^{\prime\prime}_l)\overset{S_{l+1}}{\rightsquigarrow}(M^{\prime}_{l+1},\gamma'_{l+1})\,\, ,$$
where we now  view $S_{l+1}$ as lying in $(M_l,\gamma_l)\subset (M^{\prime\prime}_l,\gamma^{\prime\prime}_l)$.  Certainly $B'_{l+1}$ satisfies conditions (1) and (2). 

Consider condition (3). We begin by considering a component $F$ of $\partial_h N(B_l^{\prime\prime})$.
The surface $F$ can be classified as one of the following 3 types (see Figure~\ref{Ftype}(b)):
\begin{enumerate}

  \item  $F$ can be viewed as a component $G$ of $\partial_hN(B_l')$, as illustrated in  Figure~\ref{Ftype}(b).  Since the new branch sectors are attached to $B_l'$ along cusp circles, $L_F$ is obtained from $L_G$ by adding curves parallel to curves in $L_G$ with coherent induced branch direction, where $L_G$ is the projection of the branch locus of $B_l'$ to $G$.  Since the branch directions are coherent, adding such parallel curves to $L_G$ does not affect the good arcs in $G$.  Thus in this case $F$ is good with respect to $B_l''$ with the same set of good arcs as $G$. \label{type1}

  \item $F$ is a horizontal boundary component for a newly created product complementary region and $\pi(F)$ contains part of the branch sectors added to $B_l'$, as illustrated in Figure~\ref{Ftype}(b).  In this case, each component of $L_F$ consists of a circle $C$ parallel to the internal boundary and with induced branch direction pointing to the internal boundary and possibly a collection of essential arcs in the annulus between $C$ and the internal boundary.   \label{type2}

  \item $F$ is in the boundary of the sutured submanifold $(M_l,\gamma_l)\subset (M^{\prime\prime}_l,\gamma^{\prime\prime}_l)$. In this case, $L_F=\emptyset$. \label{type3}
\end{enumerate}

\begin{figure}
\begin{center}
\psfrag{a}{(a)}
\psfrag{b}{(b)}
\psfrag{B}{$B_l'$}
\psfrag{C}{$B_l''$}
\psfrag{M}{$(M_l, \gamma_l)$}
\psfrag{N}{$N(B_l'')$}
\psfrag{p}{\shortstack{product\\ region}}
\psfrag{s}{\shortstack{add branch\\ sectors}}
\psfrag{1}{type (\ref{type1})}
\psfrag{2}{type (\ref{type2})}
\psfrag{3}{type (\ref{type3})}
\includegraphics[width=4.5in]{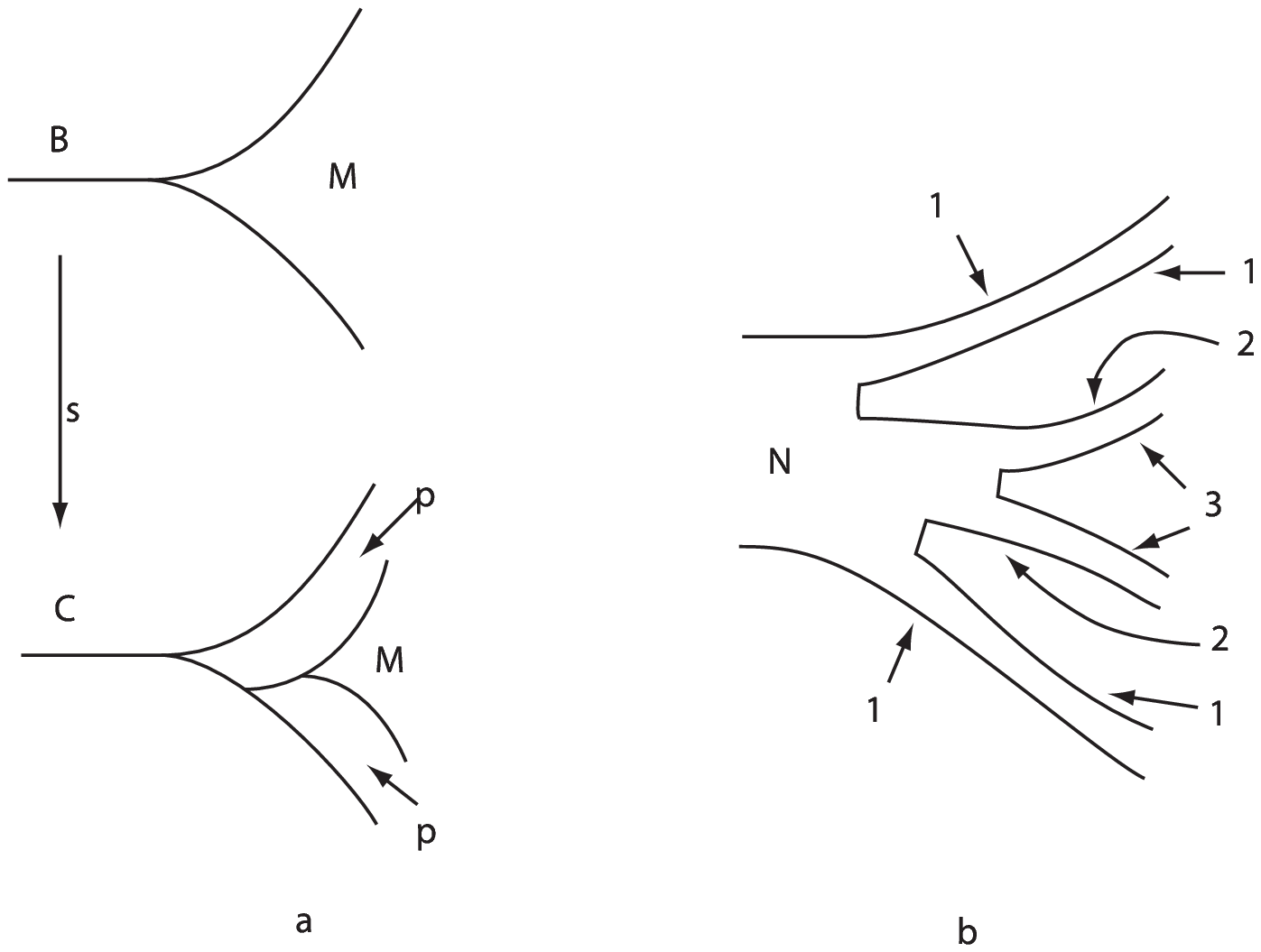}
\caption{}\label{Ftype}
\end{center}
\end{figure}

Next  consider how $\partial_hN(B_{l+1}')$ is related to  $\partial_hN(B_l'')$.
 Let $H$ be a component of $\partial_hN(B_{l+1}')$. Then either $H$ can be viewed as a component of 
$\partial_hN(B_l'')$  or $H$ contains a subset of one side of $S_{l+1}$.  Our goal is to find a set of good arcs for each component $H$ of $\partial_hN(B_{l+1}')$, so that the projections of the good arcs in $B_{l+1}'$ are disjoint.

\vspace{10pt}
\noindent
\emph{Case (a)}. $H$ is not a component of $\partial_hN(B_l'')$
\vspace{10pt}

In this case, $H$ is contained in the union of one side of $S_{l+1}$ and $F\setminus \partial S_{l+1}$, where $F$ is a component of $\partial_hN(B_l'')$ of type (3).  By our construction, $L_F=\emptyset$.  Moreover, on the other side of $F$, there is a corresponding component $F'$ of $\partial_hN(B_l'')$ of type (2) such that $\pi(F)\cap\pi(F')\ne\emptyset$ in the branched surface $B_l''$.  Adding $S_{l+1}$ to $B_l''$ does not affect $F'$, so we may also view $F'$ as a component of $\partial_hN(B_{l+1}')$.  Next we choose good arcs for both $H$ and $F'$.  

First note that since the original sutured manifold decomposition is well-groomed, $\partial S_{l+1}$ is homologically nontrivial in $H_1(F,\partial F)$.  There is a simple closed curve $\eta$ in $F$ transverse to $S_{l+1}$, as shown in Figure~\ref{Fgood} (note that the arrows in Figure~\ref{Fgood} on $\partial S_{l+1}$ denote the branch direction at $\partial S_{l+1}$), such that the algebraic intersection number of $\eta$ and $\partial S_{l+1}$ is equal to $|\eta\cap\partial S_{l+1}|$ (this is equivalent to saying that the normal direction of $\partial S_{l+1}$ at $\eta\cap\partial S_{l+1}$, induced from the branch direction of $B_{l+1}'$, are coherent along $\eta$).

Recall that $H$ can be viewed as the union of one side of $S_{l+1}$ and $F\setminus \partial S_{l+1}$.  We first consider the components $\theta_1,\dots,\theta_p$ of $\partial H\cap\partial M$ that are not in $F$ (i.e., each $\theta_i$ can be viewed as a component of $\partial S_{l+1}\cap\partial M$).  We can find an arc $\gamma_i$ connecting $\theta_i$ to the internal boundary of $H$ such that $\gamma_i$ either is totally in (one side of) $S_{l+1}$ or consists an arc in $S_{l+1}$ and an arc in $F$ parallel to a subarc of $\eta$.  Moreover, we can choose these arcs $\gamma_i$ to be disjoint in $H$.

Now we consider the components of $\partial F\cap\partial M$ (viewed as components of $\partial H\cap\partial M$).  It is easy to see from our construction that there is a collection of disjoint good arcs $\alpha_1,\dots,\alpha_q$ in $F$ (see the arcs $\alpha_1$ and $\alpha_2$ in Figure~\ref{Fgood}(a)), such that (1) these arcs $\alpha_j$ connect each component of $\partial F\cap\partial M$ to a component of $\partial S_{l+1}$, and (2) these arcs $\alpha_j$ are disjoint from the curve $\eta$ describe above.  

If follows from our construction that these arcs $\gamma_i$ and $\alpha_j$ form a set of good arcs $\Gamma_H$ for $H$.

Next we consider the component $F'$ of $\partial_hN(B_l'')$ on the other side of $F$.  So $F'$ is a type (2) component of $\partial_hN(B_l'')$, and we may view $F'$ as a component of $\partial_hN(B_{l+1}')$.  Moreover, we view $F'$ as a parallel copy of $F$ and view the curves $\partial S_{l+1}$, $\eta$ and $\alpha_j$ described above as curves in $F'$.  We have two slightly different situation.  The first is that $F'$ (and hence $F$) has nonempty internal boundary, and the second is that $F'$ has no internal boundary.

If $F'$ has nonempty internal boundary, then there are arcs $\beta_1,\dots,\beta_r$ in $F'$ (see the arcs $\beta_1$ and $\beta_2$ in Figure~\ref{Fgood}(a)), such that (1) these arcs $\beta_k$ connect each component of $\partial F'\cap\partial M$ to the internal boundary of $F'$, and (2) these arcs $\beta_k$ are disjoint from $\eta$, $\partial S_{l+1}$ and the arcs $\alpha_j$.  These arcs $\beta_k$ form a set of good arcs $\Gamma_{F'}$ for $F'$.  Moreover, since each $\beta_k$ is disjoint from $\eta$ and the arcs $\alpha_j$, the projections $\pi(\Gamma_H)$ and $\pi(\Gamma_{F'})$ of the good arcs $\Gamma_H$ and $\Gamma_{F'}$ for $H$ and $F'$ respectively are disjoint in $B_{l+1}'$.

If $F'$ does not have internal boundary (in which case $F'$ must be a Seifert surface of the knot exterior), then as shown in Figure~\ref{Fgood}(b), there is an arc $\beta$ properly embedded in $F'$ such that (1) $\beta$ is disjoint from $\eta$ and the arcs $\alpha_j$ and (2) the intersection of $\beta$ with $\partial S_{l+1}$ is minimal up to isotopy.  Since the original sutured manifold in well-groomed, the requirement (2) implies that the algebraic intersection number of $\beta$ and $\partial S_{l+1}$ is equal to $|\beta\cap\partial S_{l+1}|$.  Thus $\beta$ is a good arc for $F'$.  Since $\beta$ is chosen to be disjoint from $\eta$ and each $\alpha_j$, the projections of $\pi(\beta)$ and $\pi(\Gamma_H)$ on $B_{l+1}'$ are disjoint.

\begin{figure}
\begin{center}
\psfrag{a}{(a)}
\psfrag{b}{(b)}
\psfrag{c}{$\alpha_1$}
\psfrag{d}{$\beta$}
\psfrag{e}{$\alpha_1$}
\psfrag{f}{$\beta_1$}
\psfrag{g}{$\alpha_2$}
\psfrag{h}{$\beta_2$}
\psfrag{i}{$\eta$}
\psfrag{j}{$\partial S_{l+1}$}
\psfrag{k}{internal boundary}
\includegraphics[width=4.7in]{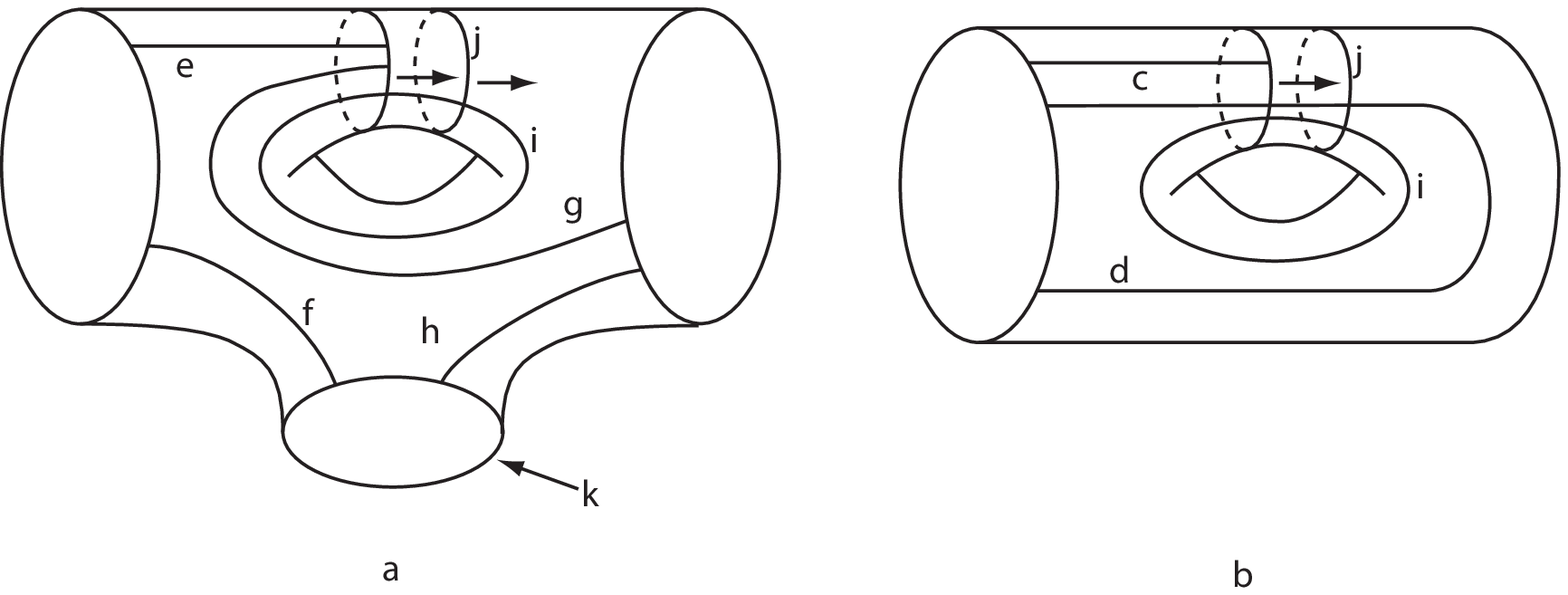}
\caption{}\label{Fgood}
\end{center}
\end{figure}

\vspace{10pt}
\noindent
\emph{Case (b)}. $H$ is a component of $\partial_hN(B_l'')$
\vspace{10pt}

In this case, either $L_H$ is unchanged by the decomposition by $S_{l+1}$ or $H$ is the surface $F'$ of type (2) considered in Case (a).  In Case (a), we have already constructed a set of good arcs for the type (2) surface $F'$, so we may assume that $L_H$ is unchanged by the decomposition by $S_{l+1}$.  Since $H$ (viewed as a component of $\partial_hN(B_l'')$) is good in $B_l''$, $H$ is good in $B_{l+1}'$.  Furthermore, the projections of the good arcs in Case (a) and the good arcs (from the induction) of $H$ in this case are disjoint in $B_{l+1}'$.

So $B_{l+1}'$ is good.
It remains to show that $B_{l+1}'$ does not carry any torus.   Since $B_l'$ does not carry any torus and $B_l''$ can be obtained by splitting $B_l'$, $B_l''$ does not carry any torus.  Suppose $B_{l+1}'$ carries a torus $T$. Then $T$ can be expressed as the union of some copies of $S_{l+1}$ and a surface in $N(B_l'')$ transverse to the $I$-fibers.  Moreover, the transverse orientation of the branched surface induces a compatible normal orientation for $T$.  Since the original sutured manifold decomposition sequence is well-groomed, $\partial S_{l+1}\cap R_\pm(\gamma_l)$ is a collection of homologically nontrivial curves in $H_1(R_\pm(\gamma_l), \partial R_\pm(\gamma_l))$.  Thus there is a component $F$ of $\partial_hN(B_l'')$, such that $T\cap F$ (with the induced orientation) is homologically nontrivial in $F$.  However, since $T$ is a torus in $S^3$, $T$ is homologically trivial and this is impossible.

Therefore, $B_{l+1}'$ satisfies properties (1)-(4) of the lemma and we can inductively construct the  sutured manifold hierarchy and corresponding sequence of branched surfaces as claimed.
\end{proof}

\subsection{Step 2: Adding product disks}\label{step2}

Let $B_n'$ be the good branched surface constructed in the proof of  Lemma~\ref{L2}.  It follows from the conditions on the sutured manifold hierarchy  and our construction above that $\partial B_n'$ consists of circles of slope $0$ in the torus $\partial M$.  In this section, we will add some product disks and modify $B_n'$ to get a laminar branched surface carrying more laminations.

As $M\setminus \Int(N(B_n'))$ is a product, we may suppose $M\setminus \Int(N(B_n'))=S\times I$, where $S$ is a compact and possibly disconnected surface.
Let $S_+=S\times \{0\}$ and $S_-=S\times\{1\}$.  So $\partial_hN(B_n')=S_+\cup S_-$. It is possible to decompose $S\times I$ as the disjoint union 
$$S\times I=(F\times I)\cup (G\times I)\, ,$$
where $F$ is the union of the components of $S$ without internal boundary.  Thus $\partial F\subset\partial M$ and each component of $G$ has nonempty internal boundary.  Moreover, each component of $F$ must be a Seifert surface in the knot exterior.  Note that, since we take parallel copies of surfaces in the horizontal boundary in each step of the sutured manifold decompositions (see Lemma~\ref{L2}), $F\ne\emptyset$.  Furthermore, $G=\emptyset$ only if $k$ is fibered.

Let $m=|\partial S_\pm\cap\partial M|$ be the number of components of the non-internal boundary $S_\pm\cap\partial M$. 
Since $B_n'$ is good, there is a collection of pairwise disjoint good arcs in $S_+$, denoted by $\eta_1,\dots, \eta_m$, and a collection of pairwise disjoint good arcs in $S_-$, denoted by $\delta_1,\dots, \delta_m$, such that 
$\pi(\bigcup_i\eta_i)\cap\pi(\bigcup_i\delta_i)=\emptyset $ (in $B_n'$) and  
each component of $\partial S_\pm\cap \partial M$ has exactly one incident good arc $\eta_i$ and one incident good arc $\delta_i$ attached to it.   
After relabeling as necessary, we may assume that $\eta_i$ and $\delta_i$, $1\le i\le r$, lie in $F\times\{0,1\}$ while $\eta_i$ and $\delta_i$, $r+1\le i\le m$, lie in $G\times\{0,1\}$.
It follows that each $\eta_i$ and each $\delta_i$, $1\le i\le r$, has both endpoints lying on $\partial M$ while each of $\eta_i$ and $\delta_i$, $r+1\le i\le m$, has exactly one endpoint lying on $\partial M$. 

Consider first $F\times [0,1]$. Recall that each component of $F$ is a Seifert surface of the knot exterior. Let $F_1$ be any component of $F$ and relabel as necessary so that
$\eta_1\subset F_1\times \{0\}$ and $\delta_1\subset F_1\times \{1\}$. By 
\cite[Lemma 4.4]{R1}, there is a sequence of simple arcs 
$$\alpha_0=\eta_1, \alpha_1,...,\alpha_l=\delta_1$$
such that $\alpha_i\cap\alpha_{i+1}=\emptyset$ and a regular neighbourhood of 
$\alpha_i\cup\alpha_{i+1}\cup\partial F_1$ in $F_1$ is a twice-punctured torus for each $i, 1\le i\le l$. 
For $1\le i\le l$, let $F_1$ induce a consistent orientation on each $F_1\times \{\frac{i}{l+1}\}$
and orient the disks $\alpha_i\times  [\frac{i}{l+1},\frac{i+1}{l+1}]$ arbitrarily.
Add branch sectors to $B_n'$ as prescribed by the following sequence of sutured manifold decompositions:
$$(M'_n,\gamma_n')\overset{A}{\rightsquigarrow} (M'_{n+1},\gamma^{\prime}_{n+1})\overset{B}{\rightsquigarrow} (M_{F_1},\gamma_{F_1})\,\, ,$$
where $A=F_1\times\{\frac{1}{l+1},...,\frac{l}{l+1}\}$ and $B=\bigcup_i (\alpha_i\times [\frac{i}{l+1},\frac{i+1}{l+1}])$. Repeat for each remaining component of $F$ and let $(M_F,\gamma_F)$ denote the resulting sutured manifold. Set $B_F=B_{(M_F,\gamma_F)}$. Notice that the conditions satisfied by the arcs $\alpha_i$ guarantee that $B_F$ is laminar.

Now consider $G\times I$. Let  $G_1$  be a component of $G$ and let $p=|\partial G_1\cap\partial M|$. 
Let $\{C_1,...,C_p\}$ be a listing of the  components of $G_1\cap\partial M$.
After relabeling as necessary, we may assume $\eta_{r+1},...,\eta_{r+p}$ lie in $G_1\times\{0\}$ and
$\delta_{r+1},...,\delta_{r+p}$ lie in $G_1\times\{1\}$, with $\{\eta_{r+i}(0),\delta_{r+i}(0)\}\subset C_i$
for each $1\le i\le p$.

\begin{lemma}
Let $\{\alpha_1,...,\alpha_p\}$ and $\{\beta_1,...,\beta_p\}$ each be a set of pairwise disjoint arcs 
properly embedded in $G_1$ with 
$\{\alpha_i(0),\beta_i(0)\}\subset C_i$ and $\{\alpha_i(1),\beta_i(1)\}\subset \partial G\setminus\{C_1,...,C_p\}$, the internal boundary of $G_1$. Let $s=|(\cup_i\alpha_i)\cap(\cup_i\beta_i)|$. Then either $s=0$ or there is a set 
$\{\gamma_1,...,\gamma_p\}$ of pairwise disjoint arcs properly embedded in $G_1$
with $\gamma_i(0)\in C_i$, $\gamma_i(1)\in\partial G_1\setminus \{C_1,...,C_p\}$, and such that
$$max\{|(\cup_i\alpha_i)\cap (\cup_i\gamma_i)|,|(\cup_i\beta_i)\cap (\cup_i\gamma_i)|\}< s\, .$$
\end{lemma}

\begin{proof}
Suppose $s\ne 0$. Relabeling as necessary, we may assume that $\alpha_1\cap(\cup_i\beta_i)\ne\emptyset$. Choose $z$ to be the point
in $\alpha_1\cap (\cup_i\beta_i)$ that is furthest along $\alpha_1$. So there are $j, t_0, t_1$ such that 
$z=\alpha_1(t_0)=\beta_j(t_1)$ and $\alpha_1(t_0,1]\cap (\cup_i\beta_i)=\emptyset$. Let $\gamma_j = \beta_j[0,t_1]\star\alpha_1[t_0,1]$ be the concatenation of the two arcs $\beta_j[0,t_1]$ and $\alpha_1[t_0,1]$, perturbed slightly so that it intersects $\alpha_1$ transversely and minimally. For $i\ne j$, set $\gamma_i=\beta_i$. Then $ |(\cup_i\alpha_i)\cap(\cup_i\gamma_i)| < |(\cup_i\alpha_i)\cap(\cup_i\beta_i)|$
and $|(\cup_i\gamma_i)\cap(\cup_i\beta_i)|= 0$. 

\end{proof}

As an immediate corollary, we have:

\begin{corollary}
There are sets of arcs $\mathcal A_i=\{\alpha_1^i,...,\alpha_p^i\}$, $1\le i\le q$, such that 
\begin{enumerate}
\item for each $i$, the arcs in $\mathcal A_i$ are pairwise disjoint and properly embedded in $G_1$,
$\alpha_j^i(0)\in C_j$, and $\alpha_j^i(1)\in \partial G_1\setminus\{C_1,...,C_p\}$, $j=1,\dots,p$,
\item $\mathcal A_0=\{\eta_{r+1},...,\eta_{r+p}\}$ and $\mathcal A_{q+1}=\{\delta_{r+1},...,\delta_{r+p}\}$,
\item for each $i$, $(\cup_j \alpha_j^i)\cap (\cup_j \alpha_j^{i+1})=\emptyset$.
\end{enumerate}
\end{corollary}

For $1\le i\le q$, let $G_1$ induce a consistent  orientation on each $G_1\times \{\frac{i}{q+1}\}$. 
Orient the disks $\alpha_j^i\times  [\frac{i}{q+1},\frac{i+1}{q+1}]$ so that the orientation induced on their boundaries agrees with the orientation of $\alpha_j^i$ (which is the orientation from its starting point in $\partial M$ to its ending point in the internal boundary). 
Add branch sectors to $B_F$ as given by the following sequence of sutured manifold decompositions:
$$(M_F,\gamma_F)\overset{A}{\rightsquigarrow} (M'_F,\gamma^{\prime}_F)\overset{B}{\rightsquigarrow} (M_{G_1},\gamma_{G_1})\,\, ,$$
where $A=G_1\times\{\frac{1}{q+1},...,\frac{q}{q+1}\}$ and $B=\bigcup_{i,j} (\alpha_j^i\times [\frac{i}{q+1},\frac{i+1}{q+1}])$. Repeat for each remaining component of $G$ and let $(M_G,\gamma_G)$ denote the resulting sutured manifold. Set $B_G=B_{(M_G,\gamma_G)}$. Notice that the conditions satisfied by the arcs $\alpha^i_j$ guarantee that $B_G$ is laminar.

By Lemma~\ref{L2}, $B_n'$ does not carry any torus.
Therefore, any branched surface obtained by splitting $B_n'$ also cannot carry a torus. 
And finally, any (closed) torus carried by $B_G$ but not this splitting of $B_n'$ would necessarily pass through one of the added disk branches and hence
 would necessarily have nonempty boundary. Thus $B_G$ does not carry a torus.

Noting that for each product disk in the above construction, its two normal directions give two ways of deforming it into a branched surface, let $B_G'$ denote the branched surface obtained from $B_G$ by reversing the orientations
of the disks $\alpha_j^i\times [\frac{i}{q+1},\frac{i+1}{q+1}]$. Notice that $B_G'$ is also laminar, has only product complementary regions, and does not carry a torus.

Hence we have laminar branched surfaces  $B_G$ and $B_G'$ with only product complementary regions and which do not carry a torus. We may therefore apply 
Theorem~\ref{T1} to conclude the existence of taut foliations realizing any boundary slope carried by $B_G\cap\partial M$ or 
$B_G'\cap\partial M$. It remains to compute these
boundary slopes.

\subsection{The boundary train tracks}

Let $\tau$ denote the train track $B_G\cap\partial M$  and let $\tau'$ denote the train track
$B_G'\cap\partial M$.

\begin{lemma}
$\tau$ and $\tau'$ together realize all slopes in $(-a,b)$ for some $a, b>0$.
\end{lemma}

\begin{proof}
Consider an annular component $A_G$ of $\partial G_1\times [\frac{i}{q+1},\frac{i+1}{q+1}]$. The train track $\tau$ (respectively $\tau'$) restricted to $A_G$ has the form indicated in Figure~\ref{ttpiece}.a  (respectively, Figure~\ref{ttpiece}.b). Similarly, consider an annular component $A_F$ of 
$\partial F_1\times [\frac{i}{l+1},\frac{i+1}{l+1}]$.  Recall that each $F_1\times \{\frac{i}{l+1}\}$ is a Seifert surface and the good arc for $F_1$ has both endpoints on the circle $\partial F_1$.  Thus both $\tau$ and $\tau'$ restricted to $A_F$ have the form indicated in Figure~\ref{ttpiece}.c.
Call all such non-longitudinal branches of $\tau$ or $\tau'$ {\it vertical}.

\begin{figure}
\begin{center}
\psfrag{a}{(a)}
\psfrag{b}{(b)}
\psfrag{c}{(c)}
\includegraphics[width=4in]{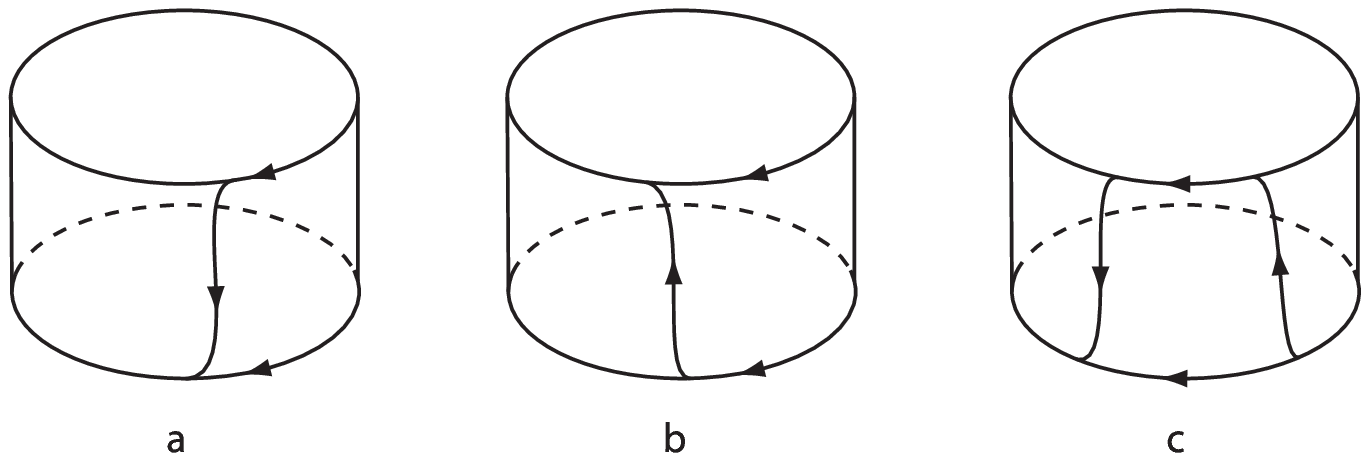}
\caption{}\label{ttpiece}
\end{center}
\end{figure}

Since all vertical branches of $\tau$ (or $\tau'$, respectively) are of one of the three types shown in Figure~\ref{ttpiece},
it follows that $\tau$ (or $\tau'$) is a train track obtained by concatenating pieces of the types of Figure~\ref{ttpiece}.a or c (or b or c, respectively).
Examples are shown in Figure~\ref{Ftrain2}. Notice that $\tau$ and $\tau'$ are orientable and measurable; namely, they admit a transverse measure (p. 66, \cite{Hatcher}, and p. 86, \cite{Penner}). Assign weights $x$, $y$, and $x+y$ to the vertical branches of $\tau$ and $\tau'$ as indicated in Figure~\ref{Ftrain2}; namely, vertical branches in $G\times I$ regions are weighted $x$, the compatibly oriented branches in $F\times I$ regions are weighted $x+y$, and the remaining branches in $F\times I$ regions are weighted $y$. 
Then assign weights from $\{1,1+x,1+y,1+x+y\}$ to the remaining branches of $\tau$ and $\tau'$ to obtain a measure $\mu$ on $\tau$ and a measure $\mu'$ on $\tau'$.

Recall that if $\gamma$ is a simple closed curve in a torus, then the slope of $\gamma$ is given in standard coordinates by
\begin{equation}\label{slope}
\mbox{slope}( \gamma)  = \frac{\langle \lambda,\gamma\rangle}{\langle \gamma,m\rangle},
\end{equation}
where $\langle ,\rangle$ denotes algebraic intersection number and $\lambda$ is the longitude and $m$ is the meridian of the knot
$k$ in $S^3$.

\begin{figure}
\begin{center}
\psfrag{x}{$x$}
\psfrag{y}{$y$}
\psfrag{z}{$x+y$}
\psfrag{1}{$1$}
\psfrag{X}{$1+x$}
\psfrag{Y}{$1+y$}
\psfrag{Z}{$1+x+y$}
\includegraphics[width=4in]{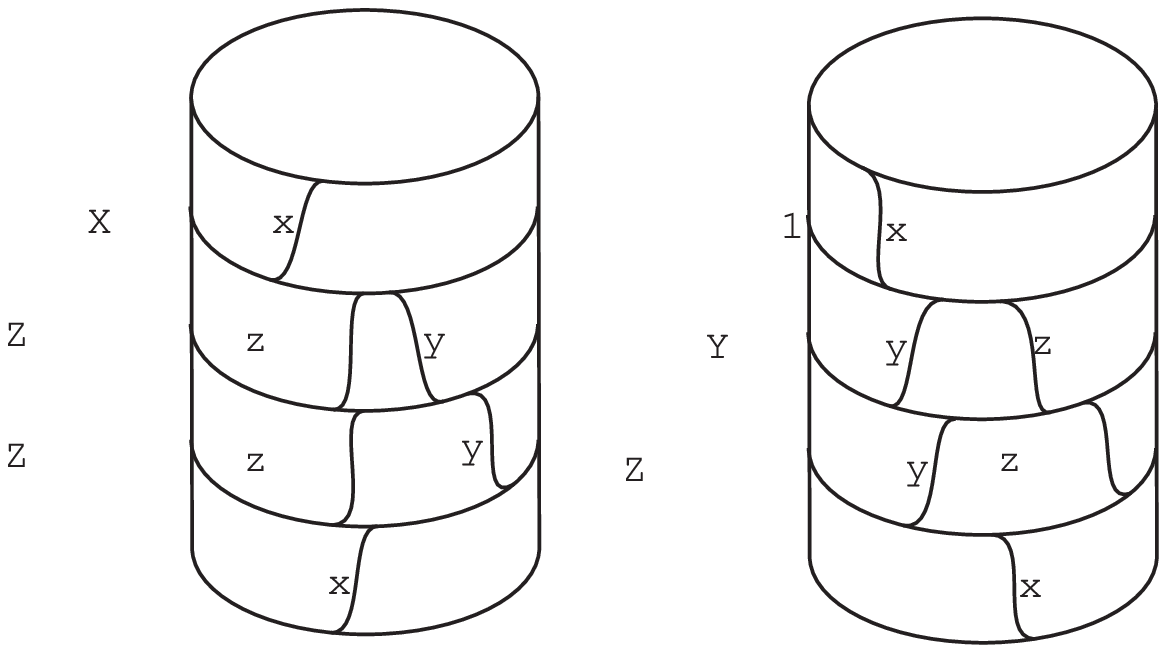}
\caption{}\label{Ftrain2}
\end{center}
\end{figure}

Applying (\ref{slope}) to the measured train tracks $(\tau,\mu)$ and $(\tau',\mu')$ while letting $x,y$ range over all values $0 < y << x$, we see that $(\tau,\mu)$ and $(\tau',\mu')$
together carry all boundary slopes in some open interval $(-a,b)$ about $0$.

\end{proof}

By Theorem~\ref{T1}, if $\tau$ (or $\tau'$) fully carries a curve of slope $s$, then $B_G$ (or $B_G'$, respectively) fully carries an essential lamination whose boundary consists of loops of slope $s$ in $\partial M$.  Moreover, this lamination extends to an essential lamination in $M(s)$.  Since $M\setminus \Int(N(B_G))$ and $M\setminus \Int(N(B_G'))$ consist of product regions, such essential laminations can be extended to taut foliations.  This proves Theorem \ref{Tmain}.

\end{document}